\documentclass[11pt, oneside]{article}   	
\usepackage{geometry}                		
\usepackage{graphicx}				
\usepackage{amssymb}
\usepackage{amsmath}
\usepackage{graphicx}				
\usepackage{amssymb}
\usepackage{amsmath}
\usepackage{amsthm}
\usepackage{harpoon}
\usepackage{accents}

\usepackage{tikz}  
\usepackage{float}
\restylefloat{table}
\usepackage{mathrsfs}
\usepackage{verbatim}

\usetikzlibrary{positioning,chains,fit,shapes,calc}  

\newtheorem{theorem}{Theorem}
\newtheorem{lemma}{Lemma}
\newtheorem{definition}{Definition}

\newtheorem{corollary}{Corollary}

\title{On $q$-analogues Arising from Elliptic Integrals and the Arithmetic-Geometric Mean}
\author{Mario DeFranco}

\begin{document}
\maketitle
\abstract{We prove $q$-analogues of identities that are equivalent to the functional equation of the arithmetic-geometric mean. We also present $q$-analogues of $F(\sqrt{k},\frac{\pi}{2})$, the complete elliptical integral of the first kind, and its derivatives 
evaluated at $k=\frac{1}{2}$. These $q$-analogues interpolate those $n$th derivative evaluations 
by extending $n$ to a complex variable $s$, and we prove that they can be expressed as an infinite product.}

\section{Introduction} \label{intro}

We present $q$-analogues arising from two closely related objects: the arithmetic-geometric mean and the complete elliptic integral of the first kind. We review these objects now. 

We recall the definition of the arithmetic-geometric mean $M(a,b)$ of two real numbers $a$ and $b$: 
Let $a_0=a$ and $b_0=b$ and define 
\[
a_{n+1} = \frac{a_n+b_n}{2}\,\,\,\,\,\mathrm{and}\,\,\,\,\,\, b_{n+1} = \sqrt{a_n b_n}.
\] 
Then 
\[
\lim_{n\rightarrow \infty} a_n = \lim_{n\rightarrow \infty} b_n = M(a,b).
\]
For information about the arithmetic-geometric mean see D. A. Cox \cite{Cox}. The properties
\[
M(ca,cb) = cM(a,b)
\]
and 
\[
M(a,b) = M(\frac{a+b}{2}, \sqrt{ab}).
\]
allow us to think of $M(a,b)$ as a function of one variable $k$ that satisfies the functional equation
\[
M(1,k) = \frac{1+k}{2}M(1,\frac{2\sqrt{k}}{1+k}).
\]
C. F. Gauss \cite{Gauss} proved that 
\[
\frac{1}{M(1,k)} = \frac{2}{\pi} \int_0^1 \frac{dt}{\sqrt{1-t^2}\sqrt{1-(1-k^2)t^2}}.
\]
The integral
\[
F(k,\frac{\pi}{2})=\int_0^1 \frac{dt}{\sqrt{1-t^2}\sqrt{1-k^2t}}
\]
is known as the complete elliptic integral of the first kind and we let $F(x)$ denote 
\[
F(x)=\frac{2}{\pi}F(\sqrt{x},\frac{\pi}{2}) = \frac{2}{\pi}\int_0^1 \frac{dt}{\sqrt{1-t^2}\sqrt{1-xt^2}} = \sum_{n=0}^\infty a_n x^n
\]
where 
\[
a_n = (\prod_{j=1}^n \frac{2j-1}{2j})^2.
\]
Therefore the functional equation in terms of $F(x)$ is 
\begin{equation} \label{functional equation F}
F(1-k^2) = \frac{2}{1+k}F((\frac{1-k}{1+k})^2).
\end{equation}
In Section \ref{identities} show that this functional equation is equivalent to a set of identities involving the $a_n$, and in Section \ref{q identities} we prove $q$-analogues of those identities. 

References that discuss the above relationship are \cite{Almqvist}, \cite{Borwein book}, \cite{Cox}, \cite{Gilmore}, \cite{Tkachev}. The proofs we have found in the literature are the three of C. F. Gauss using integral substitutions, differential equations, and another also based on the power series coefficients $a_n$. These are discussed in \cite{Cox}. There is also another proof using integrals by B. C. Carlson \cite{Carlson}. 

We now discuss how $q$-analogues enter the above discussion. We call our results ``$q$-analogues" because they involve standard expressions from $q$-theory: the $q$-positive integers
\[
[n]_q = \sum_{j=0}^{n-1} q^j = \frac{1-q^n}{1-q},
\]
the $q$-factorial
\[
n!_q = \prod_{j=1}^n [j]_q
\]
and the $q$-binomial coefficients
\[
{n \choose m}_q = \frac{n!_q}{(n-m)!_q m!_q}.
\]
We think of $q$ as an indeterminate in a formal power series or as a real number between 0 and 1. When $q=1$, the above expressions evaluate to the usual integers, factorials, and binomial coefficients. We also use the following generalizations 
\[
(\alpha)!_q = \frac{1}{(1-q)^\alpha}\prod_{n=1}^\infty \frac{1-q^n}{1-q^{\alpha+n}}
\]
and 
\[
{\alpha \choose \beta}_q = \frac{(\alpha)!_q}{(\alpha-\beta)!_q(\beta)!_q}
\]
that recover the previous formulas when $\alpha$ and $\beta$ are integers. For complex numbers $\alpha$, the $(\alpha)!_q$ is referred to as the $q$-Gamma function $\Gamma_q(\alpha)$ and satisfies 
\[
\lim_{q\rightarrow 1^-} \Gamma_q(\alpha) = \Gamma(\alpha)
\]
where $\Gamma(\alpha)$ is the Gamma function (see \cite{Andrews} for a proof). We will use this fact in Section \ref{q first elliptic}. A $q$-analogue of trigonometric functions also appears in Section \ref{q first elliptic}. 

In Section 4 we present formulas that are $q$-analogues of $\frac{d^n}{dk^n}F(\sqrt{k},\frac{\pi}{2})$ evaluated at $k=\frac{1}{2}$. We prove that these formulas are equal to an infinite  product which may be expressed using $\Gamma_q$. These equations also naturally allow the variable $n$ to take on complex values.

The $q$-formulas presented in this paper, then, may be viewed as seeking to define a $q$-analogue of the arithmetic-geometric mean, or rather a function or functions that satisfy a similar functional equation. 

Another motivation is that the Jacobi theta functions are also closely connected to the arithmetic-geometric mean and elliptic integrals (Section \ref{further work}). Furthermore, the theta functions are related to the Riemann zeta function and other Dirichlet series via the Mellin transform. Information about the arithmetic-geometric mean and elliptic integrals could thus be useful for understanding those Dirichlet series.

\section{Identities for the Functional Equation of the Arithemtic-Geometric Mean} \label{identities}

Let 
\[
F(x) = \sum_{n=0}^\infty a_n x^n.
\]
Suppose $F(x)$ satisfies the functional equation  \eqref{functional equation F}
\begin{equation*}
F(1-k^2) = \frac{2}{1+k}F((\frac{1-k}{1+k})^2).
\end{equation*}
We show this determines the $a_n$ and also evaluate the $a_n$ by setting $q=1$ in Theorem \ref{q agm functional equation identity}. 

Let 
\[
k=1-u
\]
and the functional equation \eqref{functional equation F} becomes  
\begin{equation} \label{functional equation power series}
\frac{2}{2-u}F((\frac{u}{2-u})^2) = F(u(2-u)) 
\end{equation}
and as a power series becomes
\[
\frac{2}{2-u} \sum_{n=0}^\infty a_n (\frac{u}{2-u})^{2n} =\sum_{n=0}^\infty a_n (u(2-u))^{n}.
\]
To the left side we now apply the binomial theorem 
\[
(1+x)^a = \sum_{n=0}^\infty {a \choose n} x^n
\]
and the fact 
\[
{-2n-1 \choose m} = {2n+m \choose m}(-1)^m
\]
to obtain 
\begin{align*}
\frac{2}{2-u} \sum_{n=0}^\infty a_n (\frac{u}{2-u})^{2n} &=\sum_{n=0}^\infty \frac{a_n}{2^{2n}} (\frac{u}{1-\frac{u}{2}})^{2n} \\ 
&=\sum_{n,m=0}^\infty \frac{a_n}{2^{2n+m}} u^{2n+m} {2n+m\choose 2n} \\ 
& = \sum_{k=0}^\infty \frac{u^k}{2^k} \sum_{n=0}^\infty a_n {k\choose 2n}
\end{align*}
where we have set $k = 2n+m$. 

Now the right side of \eqref{functional equation power series} becomes
\begin{align*}
\sum_{n=0}^\infty a_n (u(2-u))^{n} & = \sum_{n=0}^\infty a_n 2^nu^n(1-\frac{u}{2})^n\\ 
&= \sum_{n=0}^\infty \sum_{m=0}^n a_n 2^nu^n(-1)^{m}(\frac{u}{2})^{m} {n \choose m}\\ 
&= \sum_{k=0}^\infty \frac{u^k}{2^k}\sum_{n=0}^\infty a_n 2^{2n}(-1)^{k-n} {n \choose k-n} 
\end{align*}
where we have set $k=n+m$. 
Therefore for each integer $k \geq0$
\[
\sum_{n=0}^\infty a_n {k\choose 2n}=\sum_{n'=0}^\infty a_{n'} 2^{2n'}(-1)^{k-n'} {n' \choose k-n'}. 
\]
Now we apply the following result which we call Identity 1: 
\newline For integer $n'$ and integer $k$
\[
(-1)^{k-n'} {n' \choose k-n'} = \sum_{j=0}^\infty (-1)^j {2n'+j \choose 2n'} {k \choose n'+j}.
\]
We prove this result in Theorem \ref{q identity 1} using $q$-binomial coefficients. We get
\[
\sum_{n=0}^\infty a_n {k\choose 2n}=\sum_{n', j\geq 0} a_{n'} 2^{2n'}(-1)^j {2n'+j \choose 2n'} {k \choose n'+j}.
\]
Setting $n'+j = 2n$ gives 
\begin{equation} \label{identity 2 even}
a_n = \sum_{n'=0}^{2n} (-1)^{n'} a_{n'} 2^{2n'} {2n+n' \choose 2n'}.
\end{equation}
Setting $n'+j = 2n+1$ gives 
\begin{equation} \label{identity 2 odd}
0 = \sum_{n'=0}^{2n+1} (-1)^{n'} a_{n'} 2^{2n'} {2n+1+n' \choose 2n'}.
\end{equation}
We refer to \eqref{identity 2 even} and \eqref{identity 2 odd} as Identity 2. We present $q$-analogues of Identities 1 and 2 and prove them in Section \ref{q identities}.

\section{Proofs of $q$-analogues of Identities 1 and 2} \label{q identities}
\subsection{$q$-analogue of Identity 1}
\begin{theorem} \label{q identity 1}Let $k$ and $n$ be integers $\geq 0$. Then
\[
(-1)^{n-k}q^{\frac{(k-2n)(k-2n-1)}{2}} { n \choose k-n}_q = \sum_{j=-\infty}^{\infty} (-1)^{j}q^{\frac{(j-n)(j-n-1)}{2}} {2n+j \choose 2n }_q {k \choose n+j}_q.
\]
The sum has only finitely many non-zero terms if $k$ and $n$ are integers.  
\end{theorem}
We prove two generalizations of this result. As noted, the terms in the sum are zero if $ j>k-n$. We thus let $k = n+i$ and re-index $j \mapsto i-j$ to get 
\[
q^{\frac{(i-n)(i-n-1)}{2}} { n \choose i}_q = \sum_{j=0}^\infty (-1)^{j}q^{\frac{(i-j-n)(i-j-n-1)}{2}} {2n+i-j \choose 2n }_q {n+i \choose j}_q.
\] 
We now allow $n$ and $i$ to be possibly non-integers $s$ and $b$ satisfying certain conditions in Theorems \ref{b integer thm} and \ref{s+b integer thm}.
 \begin{theorem} \label{b integer thm}Suppose $b$ is an integer. Then
\[
q^{\frac{(b-s)(b-s-1)}{2}} { s \choose b}_q = \sum_{j=0}^b (-1)^{j}q^{\frac{(b-j-s)(b-j-s-1)}{2}} {2s+b-j \choose b-j }_q {s+b \choose j}_q.
\] 
\end{theorem}
\begin{proof}
The statement is equivalent to 
\[
\sum_{j=0}^b (-1)^j q^{j(s-b)+\frac{j(j+1)}{2}} {b \choose j}_q \prod_{k=1}^{b-j} (1-q^{2s+k}) \prod_{k=b-j+1}^b (1-q^{s+k}) = \prod_{k=1}^b (1-q^{s-k+1}). 
\]
This statement is proved in Lemma \ref{b integer lemma} for $\alpha=0$.
\end{proof}
\begin{lemma} \label{b integer lemma}
 Let $b$ be an integer. For all $\alpha$ and $s$:
\begin{align*}
&\sum_{j=0}^b (-1)^j q^{j(s-b)+\frac{j(j+1)}{2}} {b \choose j}_q \prod_{k=1}^{b-j} (1-q^{2s+\alpha+k}) \prod_{k=b-j+1}^b (1-q^{s+\alpha+k}) = \prod_{k=1}^b (1-q^{s-k+1}) \\ 
\end{align*}
\end{lemma}
\begin{proof}
Let 
\[
f(b,\alpha) = \sum_{j=0}^b (-1)^j q^{j(s-b)+\frac{j(j+1)}{2}} {b \choose j}_q \prod_{k=1}^{b-j} (1-q^{2s+\alpha+k}) \prod_{k=b-j+1}^b (1-q^{s+\alpha+k}).
\]
The lemma statement is  then 
\[
f(b,\alpha)=\prod_{k=1}^b (1-q^{s-k+1}).
\]
We use induction on $b$. The lemma is true for $b=0$. Assume it is true for some $b\geq 0$
Consider $f(b+1,\alpha)$. Use 
\[
{ b+1 \choose j}_q = q^j {b \choose j}_q+ {b \choose j-1}_q
\]
to express $f(b+1,\alpha)$ as 
\begin{align} \label{first b sum}
&\sum_{j=0}^{b} (-1)^j q^{j(s-b-1)+\frac{j(j+1)}{2}} q^j{b \choose j}_q \prod_{k=1}^{b-j+1} (1-q^{2s+\alpha+k}) \prod_{k=b-j+2}^{b+1} (1-q^{s+\alpha+k}) \\  \label{second b sum}
+&\sum_{j=1}^{b+1} (-1)^j q^{j(s-b-1)+\frac{j(j+1)}{2}} {b \choose j-1}_q \prod_{k=1}^{b-j+1} (1-q^{2s+\alpha+k}) \prod_{k=b-j+2}^{b+1} (1-q^{s+\alpha+k}). 
\end{align}
The first sum \eqref{first b sum} is equal to 
\[
(1-q^{2s+\alpha+1})f(b,\alpha+1) 
\]
and the second sum \eqref{second b sum} is equal to 
\[
 - q^{s-b}f(b,\alpha)(1-q^{s+\alpha+b+1}).
\]
Using the induction hypothesis we get that 
\[
(1-q^{2s+\alpha+1})f(b,\alpha+1) - q^{s-b}f(b,\alpha)(1-q^{s+\alpha+b+1})
\]
 is equal to 
\[
\prod_{k=1}^{b+1} (1-q^{s-k+1}).
\]
\end{proof}

\begin{theorem} \label{s+b integer thm}
Suppose $s+b=M$ is an integer $\geq 0$. Then
\[
q^{\frac{(b-s)(b-s-1)}{2}} { s \choose b}_q = \sum_{j=0}^M (-1)^{j}q^{\frac{(b-j-s)(b-j-s-1)}{2}} {2s+b-j \choose b-j }_q {s+b \choose j}_q.
\] 
\end{theorem}
\begin{proof}
The statement is equivalent to 
 \[
\sum_{j=0}^M q^{j(s+1)} { M \choose j}_q \prod_{n=1}^{M-j}(1-q^{s+n})  \prod_{n=M-j+1}^M(1-q^{s-n}) = \prod_{n=1}^M (1-q^{2s-n+1}).
\]
This statement is proved in Lemma \ref{s+b integer lemma} for $\alpha=0$.
\end{proof}

\begin{lemma}\label{s+b integer lemma}
 Let $M$ be an integer $\geq 0$. For all $\alpha$ and $s$:
 \[
\sum_{j=0}^M q^{j(s+1+\alpha)} { M \choose j}_q \prod_{n=1}^{M-j}(1-q^{s+n+\alpha})  \prod_{n=M-j+1}^M(1-q^{s-n-\alpha}) = \prod_{n=1}^M (1-q^{2s-n+1}).
\]
\end{lemma}
\begin{proof}
Let 
\[
f(M,\alpha) = \sum_{j=0}^M q^{j(s+1+\alpha)} { M \choose j}_q \prod_{n=1}^{M-j}(1-q^{s+n+\alpha})  \prod_{n=M-j+1}^M(1-q^{s-n-\alpha}).
\]
The lemma statement is 
\[
f(M,\alpha) =  \prod_{n=1}^M (1-q^{2s-n+1}).
\]
We use induction on $M$. It is true for $M=0$. Assume it is true for $M\geq 0$. Use 
\[
{ M +1\choose j}_q = q^j {M \choose j}_q+ {M \choose j-1}_q
\]
to express $f(M+1,\alpha)$ as 
\begin{align} \label{first M sum}
&\sum_{j=0}^{M} q^{j(s+1+\alpha)} q^j{ M \choose j}_q \prod_{n=1}^{M+1-j}(1-q^{s+n+\alpha})  \prod_{n=M-j+2}^M(1-q^{s-n-\alpha})\\ \label{second M sum}
+&\sum_{j=1}^{M+1} q^{j(s+1+\alpha)} { M \choose j-1}_q \prod_{n=1}^{M-j+1}(1-q^{s+n+\alpha})  \prod_{n=M-j+2}^M(1-q^{s-n-\alpha}).
\end{align}
The first sum \eqref{first M sum} is equal to 
\[
(1-q^{s+1+\alpha})f(M,\alpha+1)
\]
and the second sum \eqref{second M sum} is equal to 
\[
q^{s+1+\alpha}f(M,\alpha)(1-q^{s-M-\alpha-1}).
\]
Use the induction hypothesis to get that
\[
(1-q^{s+1+\alpha})f(M,\alpha+1,)+ q^{s+1+\alpha}f(\alpha,M)(1-q^{s-M-\alpha-1}).
\]
 this is equal to 
\[
\prod_{k=1}^{M+1} (1-q^{2s-k+1}).
\]

\end{proof}

\subsection{$q$-analogue of Identity 2}
We now present a $q$-analogue of Identity 2, the equations \eqref{identity 2 even} and \eqref{identity 2 odd}:
\begin{theorem} \label{q agm functional equation identity}
\begin{align*}
&\sum_{n=0}^m (-1)^n q^{\frac{n(n+1)}{2}-nm} (\prod_{j=1}^n \frac{1-q^{2j-1}}{1-q^{2j}})^2 \prod_{j=1}^n(1+q^{j})^2{ m+n \choose 2n}_q  \\
=&\begin{cases}   
q^{\frac{m}{2}} (\prod_{j=1}^{\frac{m}{2}} \frac{1-q^{2j-1}}{1-q^{2j}})^2 \, \text{ if $m$ is even} \\ 
0 \, \text{ if $m$ is odd}
\end{cases}
\end{align*}
\end{theorem}

We will use the following functions in a variable $c$:

\begin{definition} 

\begin{align*}
f_i(c,q) &= (1-q^{2c+2i-1})(1-q^{2c-2i+2})\\ 
F_i(c,q) &= \prod_{j=1}^{i}(1-q^{2c+2j-1})(1-q^{2c-2j+2}) = \prod_{j=1}^i f_j(c,q)\\
p_i(c,q) &=  (1-q^{2c+2i})(1-q^{2c-2i+1})\\
P_i(c,q) &= \prod_{j=1}^{i}(1-q^{2c+2j})(1-q^{2c-2j+1}) = \prod_{j=1}^i p_j(c,q)
\end{align*}

\end{definition} 

\begin{lemma} \label{2 factors} For any $x,y$ and $z$, 
\[
(1-q^x)(1-q^y) = (1-q^z)(1-q^{x+y-z}) +q^z(1-q^{x-z})(1-q^{y-z}).
\]
\end{lemma}
\begin{proof}
This is proved by straightforward calculation. 
\end{proof}
The following immediate corollaries describe two ways  we will apply Lemma \ref{2 factors}.
\begin{corollary}\label{f to p}
\[
f_i(c,q) = p_l(c,q) + q^{2c-2l+1}(1-q^{2l-2i+1})(1-q^{2i+2l-2})
\]
\end{corollary}

\begin{corollary} \label{q square}
\[
(1-q^a)^2 =(1-q^{a-k})(1-q^{a+k})+q^{a-k}(1-q^k)^2
\]
\end{corollary}

\begin{definition} For integer $m \geq 0$, define the function $I(c,m,q)$ by
\[
I(c,m,q) =  \sum_{n=0}^{m} (-1)^n q^{\frac{n(n+1)}{2}-2nc} (\prod_{j=1}^n \frac{1-q^{2j-1}}{1-q^{2j}})^2 \prod_{j=1}^n (1+q^j)^2{ n+2c\choose 2n}_q. 
\]
\end{definition} 

\begin{theorem} \label{first P factoring}For integer $l \geq 0$, 
 \begin{align*}
 &I(c,2l,q)\\
 =& (-1)^{l}q^{l^2-2lc} \frac{P_l(c,q)}{\prod_{j=1}^{2l} (1-q^j)}  \sum_{i=0}^l (-1)^i q^{i(i+1)-2ic}\frac{F_i(c,q)}{\prod_{j=1}^i (1-q^{2j})^2} \frac{\prod_{j=0}^{i-1}(1-q^{2l+2j+1})^2}{\prod_{j=0}^{2i-1} (1-q^{2l+j+1})}
\end{align*}
and 
 \begin{align*}
 &I(c,2l+1,q)\\
 =& (-1)^{l+1}q^{(l+1)^2-2(l+1)c} \frac{P_{l+1}(c,q)}{\prod_{j=1}^{2l+2} (1-q^j)}  \sum_{i=0}^l (-1)^i q^{i(i+1)-2ic}\frac{F_i(c,q)}{\prod_{j=1}^i (1-q^{2j})^2} \frac{\prod_{j=0}^{i-1}(1-q^{2l+2j+3})^2}{\prod_{j=0}^{2i-1} (1-q^{2l+j+3})}.
\end{align*}
\end{theorem}
\begin{proof}
We use induction. The statement is true for $l=0$.  Assume it is true for $l=L\geq 0$. Then we consider 
\begin{align} \nonumber{}
&(-1)^{L}q^{L^2-2Lc} \frac{P_L(c,q)}{\prod_{j=1}^{2L} (1-q^j)}  \sum_{i=0}^L (-1)^i q^{i(i+1)-2ic}\frac{F_i(c,q)}{\prod_{j=1}^i (1-q^{2j})^2} \frac{\prod_{j=0}^{i-1}(1-q^{2L+2j+1})^2}{\prod_{j=0}^{2i-1} (1-q^{2L+j+1})}  \label{first induction step} \\
+& (-1)^{2L+1} q^{\frac{(2L+1)(2L+2)}{2}-2(2L+1)c} (\prod_{j=1}^{2L+1} \frac{1-q^{2j-1}}{1-q^{2j}})^2 \prod_{j=1}^{2L+1} (1+q^j)^2{ 2L+1+2c\choose 4L+2}_q. 
\end{align}

Now 
\[
{ 2L+1+2c\choose 4L+2}_q = \frac{P_L(c,q)}{\prod_{j=1}^{2L}(1-q^{j)}}  \frac{F_{L+1}(c,q)}{\prod_{j=2L+1}^{4L+2} (1-q^j)}
\]
and 
\[
(\prod_{j=1}^{2L+1} \frac{1-q^{2j-1}}{1-q^{2j}})^2 \prod_{j=1}^{2L+1} (1+q^j)^2 = \frac{\prod_{j=0}^{L-1}(1-q^{2L+3+2j})^2}{\prod_{j=1}^{L}(1-q^{2j})^2}.
\]
Combining these we get 
\begin{align*}
&(\prod_{j=1}^{2L+1} \frac{1-q^{2j-1}}{1-q^{2j}})^2 \prod_{j=1}^{2L+1} (1+q^j)^2{ 2L+1+2c\choose 4L+2}_q\\ 
=&  \frac{P_L(c,q)}{\prod_{j=1}^{2L} (1-q^j)}\frac{F_{L+1}(c,q)}{\prod_{j=1}^{L}(1-q^{2j})^2} \frac{\prod_{j=0}^{L-1}(1-q^{2L+3+2j})^2}{\prod_{j=-2}^{2L-1}(1-q^{2L+3+j})}. 
 \end{align*}
 This allows us to express \eqref{first induction step} as 
 \begin{align*}
 \nonumber{}
&(-1)^{L}q^{L^2-2Lc} \frac{P_L(c,q)}{\prod_{j=1}^{2L} (1-q^j)} \big(  \sum_{i=0}^L (-1)^i q^{i(i+1)-2ic}\frac{F_i(c,q)}{\prod_{j=1}^i (1-q^{2j})^2} \frac{\prod_{j=0}^{i-1}(1-q^{2L+2j+1})^2}{\prod_{j=0}^{2i-1} (1-q^{2L+j+1})}  \\
+& (-1)^{L+1}q^{L(L+1) - 2Lc+2L+1-2c}\frac{F_{L+1}(c,q)}{\prod_{j=1}^{L}(1-q^{2j})^2} \frac{\prod_{j=0}^{L-1}(1-q^{2L+3+2j})^2}{\prod_{j=-2}^{2L-1}(1-q^{2L+3+j})}\big).
 \end{align*}
Now for any $0\leq h\leq L$, let $S(h)$ denote
\begin{align*}
 \nonumber{}
S(h)=&\sum_{i=0}^h (-1)^i q^{i(i+1)-2ic}\frac{F_i(c,q)}{\prod_{j=1}^i (1-q^{2j})^2} \frac{\prod_{j=0}^{i-1}(1-q^{2L+2j+1})^2}{\prod_{j=0}^{2i-1} (1-q^{2L+j+1})}  \\
+& (-1)^{h+1}q^{h(h+1) - 2hc+2L+1-2c}\frac{F_{h+1}(c,q)}{\prod_{j=1}^{h}(1-q^{2j})^2} \frac{\prod_{j=0}^{h-1}(1-q^{2L+3+2j})^2}{\prod_{j=-2}^{2h-1}(1-q^{2L+3+j})}.\\
\end{align*}
For $h \geq 1$, we claim
\[
S(h) = S(h-1) - \frac{q^{2L+1-2c}p_{L+1}(c,q)}{(1-q^{2L+1})(1-q^{2L+2}))} (-1)^h q^{h(h+1)-2hc}\frac{F_h(c,q)}{\prod_{j=1}^h (1-q^{2j})^2} \frac{\prod_{j=0}^{h-1}(1-q^{2L+2j+1})^2}{\prod_{j=0}^{2h-1} (1-q^{2L+j+1})}.
\]
This follows from taking the $h$-th and $(h+1)$-th term in $S(h)$ and first applying Corollary \ref{f to p} for $f_{h+1}(c,q)$ and $p_{L+1}(c,q)$; and then Corollary \ref{q square} for $a=2L+1$ and $k = 2h$: 

\begin{align*} 
&(-1)^h q^{h(h+1)-2hc}\frac{F_h(c,q)}{\prod_{j=1}^h (1-q^{2j})^2} \frac{\prod_{j=0}^{h-1}(1-q^{2L+2j+1})^2}{\prod_{j=0}^{2h-1} (1-q^{2L+j+1})} \\
+&(-1)^{h+1}q^{h(h+1) - 2hc+2L+1-2c}\frac{F_{h+1}(c,q)}{\prod_{j=1}^{h}(1-q^{2j})^2} \frac{\prod_{j=0}^{h-1}(1-q^{2L+3+2j})^2}{\prod_{j=-2}^{2h-1}(1-q^{2L+3+j})}\\ 
&= (-1)^{h}q^{(h-1)h - 2(h-1)c+2L+1-2c}\frac{F_{h}(c,q)}{\prod_{j=1}^{h-1}(1-q^{2j})^2} \frac{\prod_{j=0}^{h-2}(1-q^{2L+3+2j})^2}{\prod_{j=-2}^{2h-3}(1-q^{2L+3+j})}\\ 
-& \frac{q^{2L+1-2c}p_{L+1}(c,q)}{(1-q^{2L+1})(1-q^{2L+2}))} (-1)^h q^{h(h+1)-2hc}\frac{F_h(c,q)}{\prod_{j=1}^h (1-q^{2j})^2} \frac{\prod_{j=0}^{h-1}(1-q^{2L+2j+1})^2}{\prod_{j=0}^{2h-1} (1-q^{2L+j+1})}
\end{align*}
By the same reasoning we check
\[
S(0) = -\frac{q^{2L+1-2c}p_{L+1}(c,q)}{(1-q^{2L+1})(1-q^{2L+2}))}. 
\]
Now use 
\[
(-1)^{L}q^{L^2-2Lc} \frac{P_L(c,q)}{\prod_{j=1}^{2l} (1-q^j)} \left (-\frac{q^{2L+1-2c}p_{L+1}(c,q)}{(1-q^{2L+1})(1-q^{2L+2})} \right )= (-1)^{L+1}q^{(L+1)^2-2(L+1)c} \frac{P_{L+1}(c,q)}{\prod_{j=1}^{2L+2} (1-q^j)} 
\]
to get 

\begin{align*}
&I(c,2L+1,q)\\
 =& (-1)^{L+1}q^{(L+1)^2-2(L+1)c} \frac{P_{L+1}(c,q)}{\prod_{j=1}^{2L+2} (1-q^j)}  \sum_{i=0}^L (-1)^i q^{i(i+1)-2ic}\frac{F_i(c,q)}{\prod_{j=1}^i (1-q^{2j})^2} \frac{\prod_{j=0}^{i-1}(1-q^{2L+3+2j})^2}{\prod_{j=0}^{2i-1} (1-q^{2L+3+j})}.
\end{align*}
This completes the part of the theorem for $I(c,2L+1,q)$. To this we add 
\begin{align*}
&(-1)^{2L+2} q^{\frac{(2L+2)(2L+3)}{2}-2(2L+2)c} (\prod_{j=1}^{2L+2} \frac{1-q^{2j-1}}{1-q^{2j}})^2 \prod_{j=1}^{2L+2} (1+q^j)^2{ 2L+2+2c\choose 4L+4}_q \\
& =  (-1)^{L+1}q^{(L+1)^2-2(L+1)c} \frac{P_{L+1}(c,q)}{\prod_{j=1}^{2L+2} (1-q^j)}  \\ 
&\times \left((-1)^{L+1} q^{(L+1)(L+2)-2(L+1)c} \frac{F_{L+1}(c,q)}{\prod_{j=1}^{L+1} (1-q^{2j})^2} \frac{\prod_{j=0}^{L}(1-q^{2L+3+2j})^2}{\prod_{j=0}^{2L+1} (1-q^{2L+3+j})}. \right)\\
\end{align*}
which completes the part of the theorem for $I(c,2L+2,q)$.
\end{proof}
\begin{corollary}
Theorem \ref{q agm functional equation identity} is true for the case of odd $m$.  
\end{corollary}
\begin{proof}
Theorem \ref{first P factoring} shows that $I(c,2l+1,q)$ has a factor of $P_{l+1}(c,q)$. Evaluating at $c= l+\frac{1}{2}$ yields $P_{l+1}(l+\frac{1}{2})=0$.  
\end{proof}
We introduce the variable $a$: 
\begin{definition} For integer $l\geq 0$, define the function $G(c,a,l,q)$
\[
G(c,a,l,q)= \sum_{i=0}^l (-1)^i 
 q^{(c-i)(c-i-1)} \frac{F_i(c,q)}{\prod_{j=1}^i (1-q^{2j})^2} \frac{\prod_{j=0}^{i-1}(1-q^{a+2j})^2}{\prod_{j=0}^{2i-1} (1-q^{a+j})}.
\]
\end{definition}
With this function we can express Theorem \ref{first P factoring} as 
\begin{align*}
 &I(c,2l,q)\\
 =& (-1)^{l}q^{l^2-2lc} \frac{P_l(c,q)}{\prod_{j=1}^{2l} (1-q^j)}  q^{-c^2+c} G(c,l,2l+1,q) 
\end{align*}
and 
 \begin{align*}
 &I(c,2l+1,q)\\
 =& (-1)^{l+1}q^{(l+1)^2-2(l+1)c} \frac{P_{l+1}(c,q)}{\prod_{j=1}^{2l+2} (1-q^j)}   q^{-c^2+c} G(c,l,2l+3,q)
\end{align*}
We now evaluate $G(c,a,l,q)$ in terms of the $P_i(c,q)$: 
\begin{theorem} \label{second P factoring}
\[
G(c,a,l,q) = \sum_{i=0}^l (-1)^i q^{(c-i)(c-i-1)}\frac{P_i(c,q)}{\prod_{j=1}^i (1-q^{2j})^2} (\frac{1-q^a}{1-q^{a+2i}})(\prod_{j=1}^{l-i}\frac{1-q^{2j-1}}{1-q^{2j}})  (\prod_{j=1}^{l}\frac{1-q^{a+2j}}{1-q^{a+2j-1}}) 
\]
\end{theorem}

To prove this we first express the $F_i(c,q)$ in terms of the $P_i(c,q)$:
\begin{lemma} \label{F to P}For integer $h \geq 0$, 
\[
F_h(c,q) = \sum_{v=0}^{h} q^{2cv}P_{h-v}(c,q) F_{v}(-1,q) q^{2v(1-h+v)}\prod_{k=1}^{h-v} \frac{(1-q^{2k+2v})^2}{(1-q^{2k})^2}
\]
\end{lemma} 
\begin{proof}
We use induction. The statement is true for $h=0$. Assume it is true for an $h\geq 0$. 
We then multiply both sides by $f_{h+1}(c,q)$. To each $f_{h+1}(c,q)P_{h-v}(c,q)$ on the right side we apply Corollary \ref{f to p} obtain 

\[
f_{h+1}(c,q)P_{h-v}(c,q) = P_{h-v+1}(c,q) + q^{2c-2(h-v+1)+1}(1-q^{-2v+1})(1-q^{4h-2v+2})P_{h-v}(c,q). 
\]
Then we collect terms to equate the coefficient of $P_{h+1-v}(c)$ for each $0 \leq v \leq h+1$ with the coefficient in the lemma. For $v=0$, we get 
\[
1=1.
\]
For $1 \leq v \leq h+1$, we get
\begin{align*}
&F_v(-1) q^{2v(v-h+1)+2cv} \prod_{k=1}^{h-v} \frac{(1-q^{2k+2v})^2}{(1-q^{2k})^2}\\ 
+&q^{2c-1-2(h-v+1)}(1-q^{2(h-v+1)-2h+1})(1-q^{2(h-v+1)+2h+2})F_{v-1}(-1)q^{2(v-1)(v-h)+2c(v-1)}\prod_{k=1}^{h-v+1} \frac{(1-q^{2k+2(v-1)})^2}{(1-q^{2k})^2}\\ 
=&F_v(-1) q^{2v(v-h)+2cv} \prod_{k=1}^{h-v+1} \frac{(1-q^{2k+2v})^2}{(1-q^{2k})^2}.
\end{align*}
The above equation is implied by the following equation 
\begin{align*}
&(1-q^{2v-3})(1-q^{-2v})q^{2v(v-h+1)+2 c v} \\ 
+&(1-q^{4h-2v+4})(1-q^{-2v+3})(\frac{1-q^{2v}}{1-q^{2h-2v+2}})^2 q^{2(v-1)(v-h)+2c(v-1)+2c-1-2(h-v+1)}\\ 
=&(1-q^{2v-3})(1-q^{-2v})(\frac{1-q^{2h+2}}{1-q^{2h-2v+2}})^2 q^{2v(v-h)+2 c v}.
\end{align*}
The above equation reduces to the following which is an instance of Corollary \ref{q square}: 
\[
(1-q^{2+2h-2v})^2 = (1-q^{4+4h-2v})(1-q^{-2v})+ q^{-2v}(1-q^{2+2h})^2.
\]
This completes the proof.
\end{proof}

Now we can prove Theorem \ref{second P factoring}:
\begin{proof}
We use induction on $l$. The statement is true for $l=0$. Assume it is true for some $l \geq 0$. Then we must show 
\begin{align*}
& \sum_{i=0}^l (-1)^i q^{(c-i)(c-i-1)}\frac{P_i(c,q)}{\prod_{j=1}^i (1-q^{2j})^2} (\frac{1-q^a}{1-q^{a+2i}})(\prod_{j=1}^{l-i}\frac{1-q^{2j-1}}{1-q^{2j}})  (\prod_{j=1}^{l}\frac{1-q^{a+2j}}{1-q^{a+2j-1}}) \\
+& (-1)^i 
 q^{(c-l-1)(c-l-2)} \frac{F_{l+1}(c,q)}{\prod_{j=1}^{l+1} (1-q^{2j})^2} \frac{\prod_{j=0}^{l}(1-q^{a+2j})^2}{\prod_{j=0}^{2l+1} (1-q^{a+j})}\\ 
 &= \sum_{i=0}^{l+1} (-1)^i q^{(c-i)(c-i-1)}\frac{P_i(c,q)}{\prod_{j=1}^i (1-q^{2j})^2} (\frac{1-q^a}{1-q^{a+2i}})(\prod_{j=1}^{l+1-i}\frac{1-q^{2j-1}}{1-q^{2j}})  (\prod_{j=1}^{l+1}\frac{1-q^{a+2j}}{1-q^{a+2j-1}}). 
\end{align*}
We apply Lemma \ref{F to P} to $F_{l+1}(c,q)$  and equate the coefficient of $P_i(c,q)$ to the that in the Theorem to obtain for $0 \leq i \leq l$:

\begin{align*}
&(-1)^i q^{(c-i)(c-i-1)} \frac{1}{\prod_{j=1}^i (1-q^{2j})^2} \frac{(1-q^a)}{(1-q^{a+2i})} \prod_{j=1}^{l-i}\frac{(1-q^{2j-1})}{(1-q^{2j})}\prod_{j=1}^{l}\frac{(1-q^{a+2j})}{(1-q^{a+2j-1})}\\ 
+&(-1)^{l+1} q^{(c-l-1)(c-l-2)+2(l+1-i)(1-i)+2c(l+1-i)} \frac{1}{\prod_{j=1}^{l+1} (1-q^{2j})^2} F_{l+1-i}(-1)\\
\times & \prod_{j=0}^{l} \frac{(1-q^{a+2j})}{(1-q^{a+2j+1})}\prod_{k=1}^i \frac{(1-q^{2k+2(l+1-i)})^2}{(1-q^{2k})^2}\\ 
=&(-1)^i q^{(c-i)(c-i-1)} \frac{1}{\prod_{j=1}^i (1-q^{2j})^2} \frac{(1-q^a)}{(1-q^{a+2i})} \prod_{j=1}^{l+1-i}\frac{(1-q^{2j-1})}{(1-q^{2j})}\prod_{j=1}^{l+1}\frac{(1-q^{a+2j})}{(1-q^{a+2j-1})}.
\end{align*} 

The above equation is implied by the following equation:

\begin{align*}
&q^{(c-i)(c-i-1)}\frac{(1-q^a)}{1-q^{a+2i}}\\ 
+ &q^{(c-l-1)(c-l-2)+2(l+1-i)(1-i)+2c(l+1-i)-(l+1-i)(l+2-i)}\frac{(1-q^{-1})(1-q^a)}{(1-q^{2l+2-2i})(1-q^{a+2l+1})}\\ 
=&q^{(c-i)(c-i-1)}\frac{(1-q^a)(1-q^{a+2l+2})(1-q^{2l-2i+1})}{(1-q^{a+2i})(1-q^{a+2l+1})(1-q^{2l-2i+2})}.
\end{align*}
This reduces to 
\begin{equation} \label{reduction}
(1-q^{2+a+2l})(1-q^{1-2i+2l}) = (1-q^{1+a+2l})(1-q^{2-2i+2l})+q^{2-2i+2l}(1-q^{-1})(1-q^{a+2i})
\end{equation}
which is an instance of Lemma \ref{2 factors}.

When $l+1=i$ the equation between the coefficients is 

\begin{align*}
&(-1)^{l+1} q^{(c-l-1)(c-l-2)+2(l+1-i)(1-i)+2c(l+1-i)} \frac{1}{\prod_{j=1}^{l+1} (1-q^{2j})^2} F_{l+1-i}(-1)\\
\times & \prod_{j=0}^{l} \frac{(1-q^{a+2j})}{(1-q^{a+2j+1})}\prod_{k=1}^i \frac{(1-q^{2k+2(l+1-i)})^2}{(1-q^{2k})^2}\\ 
=&(-1)^i q^{(c-i)(c-i-1)} \frac{1}{\prod_{j=1}^i (1-q^{2j})^2} \frac{(1-q^a)}{(1-q^{a+2i})} \prod_{j=1}^{l+1-i}\frac{(1-q^{2j-1})}{(1-q^{2j})}\prod_{j=1}^{l+1}\frac{(1-q^{a+2j})}{(1-q^{a+2j-1})}.
\end{align*}

This is implied by the following equation: when $l+1=i$
\[
F_{l+1-i}(-1) \frac{\prod_{k=1}^i (1-q^{2k+2l-2i+2})^2}{\prod_{j=1}^{l+1}(1-q^{2j})^2} =(-1)^{l+i-1}q^{(l+1-i)(l+2-i)} \frac{(1-q^{-1})\prod_{j=1}^{l-i} (1-q^{2j-1})}{(1-q^{2l-2i+2})\prod_{j=1}^{l-i} (1-q^{2j})}
\]
also reduces to \eqref{reduction} for $l+1=i$. This completes the proof.
\end{proof}
We use Theorem \ref{second P factoring} to evaluate $G(c,a,l,q)$ at $c=l$:
\begin{theorem} \label{evaluate G}
\[
G(l,a,l,q) = (-1)^l q^{l(l-1)}\prod_{j=1}^{l}\frac{(1-q^{2j-1})(q^{2j}-q^a)}{(1-q^{2j})(1-q^{a+2j-1})}.
\]
\end{theorem}
\begin{proof}
Theorem \ref{second P factoring} expresses $G(c,a,l,q)$ as a function of $a$ using the Lagrange interpolation form of a polynomial. That is, for $0 \leq i \leq l$ and $a=-2i$, each term in the sum is 0 except for the $i$-th term. Therefore we can easily evaluate $G(c,-2i,l,q)$ as a factored expression. After multiplying both sides of this theorem statement by 
\[
\prod_{j=1}^l (1-q^{a+2j-1}),
\]   
both sides are polynomials in $q^a$ of degree at most $l$. Therefore if they agree at $a=-2i$ for $0 \leq i \leq l$, then they are equal as functions of $a$. 
We get 
\begin{align*}
&\prod_{j=1}^l (1-q^{2j-2i-1})G(l,-2i,l,q) \\
=& (-1)^i q^{(l-i)(l-i-1)}\frac{P_i(l,q)}{\prod_{j=1}^i (1-q^{2j})^2}(\prod_{j=1}^{l-i}\frac{1-q^{2j-1}}{1-q^{2j}})  \prod_{j=0}^{i-1}(1-q^{-2i+2j}) \prod_{j=i+1}^{l}(1-q^{-2i+2j}).
\end{align*}
This simplifies to 
\begin{align*}
&q^{(l-i)(l-i-1)-i(i+1)}\frac{\prod_{j=1}^i (1-q^{2j}) \prod_{j=1}^i (1-q^{2l+2j}) \prod_{j=1}^{l-i} (1-q^{2j})\prod_{j=1}^l (1-q^{2j-1})}{\prod_{j=1}^i (1-q^{2j})^2  \prod_{j=1}^{l-i} (1-q^{2j})}\\ 
&= q^{(l)(l-1)-2il} \frac{ \prod_{j=1}^l (1-q^{2j-1}) \prod_{j=1}^i (1-q^{2l+2j}) }{\prod_{j=1}^i (1-q^{2j}) }.
\end{align*}
And 
\[
 (-1)^l q^{l(l-1)}\prod_{j=1}^{l}\frac{(1-q^{2j-1})(q^{2j}-q^{-2i})}{(1-q^{2j})}
\]
simplifies to the same thing. This completes the proof.
\end{proof}

Now we can prove Theorem \ref{q agm functional equation identity} in the case when $m$ is even: 
\begin{proof}
 Let $m=2l$. Combining Theorems \ref{first P factoring}, \ref{second P factoring}, and \ref{evaluate G}, we evaluate $c=l$ and $a=2l+1$ to get 
 \begin{align*}
& \sum_{n=0}^{2l} (-1)^n q^{\frac{n(n+1)}{2}-2nl} (\prod_{j=1}^n \frac{1-q^{2j-1}}{1-q^{2j}})^2 \prod_{j=1}^n(1+q^{j})^2{ 2l+n \choose 2n}_q\\ 
=&q^{l-l^2}\left((-1)^l q^{-l^2} \frac{P_{L}(l,q)}{\prod_{j=1}^{2l} (1-q^j)} \right)\left((-1)^l q^{l(l-1)}\prod_{j=1}^{l}\frac{(1-q^{2j-1})(q^{2j}-q^{2l+1})}{(1-q^{2j})(1-q^{2l+2j})}\right).
 \end{align*}
 This simplifies to 
 \[
 q^{l}\frac{\prod_{j=1}^l (1-q^{2j-1})^3 (1-q^{2l+2j})}{\prod_{j=1}^l (1-q^{2j})^2 (1-q^{2j-1}) (1-q^{2l+2j})} = q^l \prod_{j=1}^l(\frac{ 1-q^{2j-1}}{ 1-q^{2j} } )^2
 \]
 which completes the proof.
\end{proof} 

We include a result when $q=1$:
\begin{lemma} For integer $m \geq 0$, the following functions of $c$ are all equal:
\[
H(c,m)=\left(\prod_{u=1}^{m} (1+\frac{2c}{2u})(1-\frac{2c}{2u-1}) \right )\sum_{i=0}^\infty(-1)^i {m+\frac{1}{2} \choose i}  \frac{ F_i(c,1)}{F_i(-m-1,1)}
\]
where 
\[
\frac{ F_i(c,1)}{F_i(-m-1,1)} = \prod_{j=1}^i \frac{(2c+2j-1)(2c+2-2j)}{(-2m-3+2j)(-2m-2j)}.
\]
\end{lemma}

\begin{proof}
We prove that $H(c,m) = H(c,m+1)$ by showing 
\begin{equation} \label{m+1 step}
\sum_{i=0}^\infty(-1)^i {m+\frac{1}{2} \choose i}  \frac{ F_i(c,1)}{F_i(-m-1,1)} = (1+\frac{2c}{2m+2})(1-\frac{2c}{2m+1}) \sum_{i=0}^\infty(-1)^i {(m+1)+\frac{1}{2} \choose i}  \frac{ F_i(c,1)}{F_i(-(m+1)-1,1)}.
\end{equation}
The sum in the lemma for fixed $m$ and $c$ is absolutely convergent, 
 as the product 
 \[
 \prod_{j=1}^i \frac{(2c+2j-1)(2c+2-2j)}{(-2m-3+2j)(-2m-2j)}
 \]
is convergent as $i \rightarrow \infty$ and 
\[
|{x\choose i} | \leq  \frac{C_x}{i^{x+1}}
\]
as $i \rightarrow \infty$ where $C_x$ is a constant that depends on $x$. We have
\[
\sum_{i=0}^N (-1)^i {m+\frac{1}{2} \choose i} =  \prod_{j=1}^N \frac{2j-2m-1}{2j}.
\]
This follows from 
\[
\sum_{i=0}^N (-1)^i q^{i(i+1)-i(2m+1)} {m+\frac{1}{2} \choose i}_{q^2} =  \prod_{j=1}^N \frac{1-q^{2j-2m-1}}{1-q^{2j}}
\]
which can proved by induction. We denote 
\[
W(N) = \prod_{j=1}^N \frac{2j-2m-1}{2j}.
\]
From Corollary \ref{f to p}, we have 
\[
(2c-1+2i)(2c+2-2i)=(2c+2m+2)(2c-1-2m)+(-2m-3+2i)(-2m-2i).
\]
This implies 
\[
\frac{ F_i(c,1)}{F_i(-m-1,1)} = 1-(1+\frac{2c}{2m+2})(1-\frac{2c}{2m+1})  \sum_{j=0}^{i-1} \frac{(-2m-3)(-2m-1)}{(-2m+2j-3)(-2m+2j-1)}\frac{ F_j(c,1)}{F_j(-m-2,1)}.
\]
Using 
\[
\frac{(-2m-3)(-2m-1)}{(-2m+2j-3)(-2m+2j-1)} \prod_{i=1}^j \frac{2i-2m-1}{2i} = (-1)^j {(m+1) +\frac{1}{2} \choose j}
\]
we obtain for any $N>0$
\begin{align}
&\sum_{i=0}^\infty  (-1)^i {m+\frac{1}{2} \choose i}  \frac{ F_i(c,1)}{F_i(-m-1,1)}\nonumber \\
= & (1+\frac{2c}{2m+2})(1-\frac{2c}{2m+1})\sum_{i=0}^{N-1} (-1)^i {(m+1)+\frac{1}{2} \choose i}  \frac{ F_i(c,1)}{F_i(-(m+1)-1,1)} \label{m+1 function}  \\
 +&\sum_{i=N+1}^\infty (-1)^i {(m+1)+\frac{1}{2} \choose i}  \frac{ F_i(c,1)}{F_i(-(m+1)-1,1)}  \label{tail} \\ 
 +&W(N)\left(1- (1+\frac{2c}{2m+2})(1-\frac{2c}{2m+1})\right)\sum_{i=0}^{N-1} \frac{(-2m-3)(-2m-1)}{(-2m+2i-3)(-2m+2i-1)} \frac{ F_i(c,1)}{F_i(-m-1,1)} \label{W}. 
\end{align}
Now as $N\rightarrow\infty$,   the expression \eqref{m+1 function} goes to the right side of \eqref{m+1 step}; expression \eqref{tail} goes to 0; and expression \eqref{W} goes to 0 because the sum is convergent and $ \displaystyle \lim_{N\rightarrow\infty} W(N) =0$.

\end{proof}
We note that this Corollary is sufficient to prove Identity 2 for $q=1$.
\begin{corollary}
\[
I(c,\infty,1)=\sum_{n=0}^{\infty} (-1)^n  (\prod_{j=1}^n \frac{2j-1}{2j})^2 2^{2n}{ n+2c\choose 2n} = \prod_{u=1}^{\infty} (1+\frac{2c}{2u})^2(1-\frac{2c}{2u-1})^2 
\]
\end{corollary}
\begin{proof}
By Theorem \ref{first P factoring}, we have 

\[
I(c,2l,1) = \prod_{u=1}^{l} (1+\frac{2c}{2u})(1-\frac{2c}{2u-1}) \sum_{i=0}^l (-1)^i A(i,l){\frac{1}{2} \choose i}  \frac{ F_i(c,1)}{F_i(-1,1)}
\]
where 
\[
A(i,l) = \prod_{j=1}^i \frac{2l+2j-1}{2l+2j}.
\]
For each $i>0$,  $A(i,l)$ is an increasing function in $l$ and approaches 1 as $l\rightarrow \infty$. By the absolute convergence mentioned in the lemma for $H(c,0)$, we have that 
\[
\lim_{l\rightarrow \infty} I(c,2l,1) = \left(\prod_{u=1}^{\infty} (1+\frac{2c}{2u})(1-\frac{2c}{2u-1}) \right)H(c,0).
\]  
And
\[
\lim_{m \rightarrow \infty } H(c,m) = \prod_{u=1}^{\infty} (1+\frac{2c}{2u})(1-\frac{2c}{2u-1}) 
\]
because in the sum in the lemma, the term for $i>0$
\[
{m+\frac{1}{2} \choose i}  \frac{ F_i(c,1)}{F_i(-m-1,1)} = \prod_{j=1}^i \frac{(2c+2j-1)(2c-2j+2)}{(2j)(2m+2j)}
\]
is decreasing in magnitude to 0 for fixed $c$ as $m \rightarrow \infty$ and remains the constant 1 if $i=0$. By the absolute convergence of the sum the limit is therefore 1. 

\end{proof}

\subsection{Trying to Reconcile Identities 1 and 2}

Recall that the functional equation for the arithmetic-geometric mean is equivalent to
\[
\sum_{n=0}^\infty a_n {k\choose 2n}=\sum_{n'=0}^\infty a_{n'} 2^{2n'}(-1)^{k-n'} {n' \choose k-n'}. 
\]
for each integer $k \geq0$. We set 
\[
a_n(q) = \prod_{j=1}^n \frac{1-q^{2j-1}}{1-q^{2j}}
\]
and therefore write 
\begin{equation} \label{q agm binomial k}
\sum_{n=0}^\infty a_n(q) q^{f_1(n,k)}{k\choose 2n}_q=\sum_{n'=0}^\infty a_{n'}(q) \prod_{j=1}^n (1+q^j)^2 (-1)^{k-n'} q^{f_2(n',k)} {n' \choose k-n'}_q. 
\end{equation}
where $f_1(n,k)$ and $f_2(n,k)$ are functions  on $\mathbb{N}_0^2$ we will try to determine. To the above equation we apply Identity 1:
for integer $n'$ and integer $k$
\begin{equation} \label{q identity 1 again}
(-1)^{n'-k}q^{\frac{(k-2n')(k-2n'-1)}{2}} { n' \choose k-n'}_q = \sum_{j=0}^{\infty} (-1)^{j}q^{\frac{(j-n')(j-n'-1)}{2}} {2n'+j \choose 2n' }_q {k \choose n'+j}_q.
\end{equation}
We get 
\begin{align*}
&\sum_{n=0}^\infty a_n(q) q^{f_1(2n,k)}{k\choose 2n}_q\\
=&\sum_{n'=0}^\infty\sum_{j=0}^\infty  a_{n'}(q) \prod_{j=1}^n (1+q^j)^2 (-1)^j {2n'+j \choose 2n'}_q {k \choose n'+j}_q q^{f_2(n',k)-\frac{(k-2n')(k-2n'-1)}{2}+ \frac{(j-n')(j-n'-1)}{2}}.
\end{align*}
Setting $n'+j = 2n$ gives 
\begin{equation} \label{identity even}
a_n(q)q^{f_1(2n,k)} = \sum_{n'=0}^{2n} (-1)^{n'} a_{n'}(q) \prod_{j=1}^n (1+q^j)^2 {2n+n' \choose 2n'}q^{f_2(n',k)-\frac{(k-2n')(k-2n'-1)}{2}+ \frac{(2n-2n')(2n-2n'-1)}{2}}.
\end{equation}
Now with $m=2n$, Identity 2 is
\[
\sum_{n'=0}^{2n} (-1)^{n'} q^{\frac{n'(n'+1)}{2}-2n'n}a_{n'}(q) \prod_{j=1}^{n'}(1+q^{j})^2{ 2n+n' \choose 2n'}_q =   
q^{n} a_n(q).
\]
Therefore we have 
\[
n-f_1(2n,k) +f_2(n',k)-\frac{(k-2n')(k-2n'-1)}{2}+ \frac{(2n-2n')(2n-2n'-1)}{2} = \frac{n'(n'+1)}{2}-2n'n.
\]
For $k=4$, we therefore have a system of nine equations that come from the nine possible values for $(n,n')$ such that $0 \leq n' \leq 2n$ and ${k \choose 2n} \neq 0$: 
\[
(0,0), (1,0), (1,1),(1,2), (2,0),(2,1), (2,2), (2,3), (2,4).
\]
These nine equations are in the eight variables 
\[
f_1(0,4), f_1(2,4), f_1(4,4), f_2(0,4), f_2(1,4), f_2(2,4), f_2(3,4), f_2(4,4)
\] 
and we check that the system has no solution. 

Setting $n'+j = 2n+1$ gives 
\begin{equation} \label{identity odd}
0 = \sum_{n'=0}^{2n+1} (-1)^{n'} a_{n'}(q) \prod_{j=1}^{n'}(1+q^{j})^2 {2n+1+n' \choose 2n'}_q q^{f_2(n',k)-\frac{(k-2n')(k-2n'-1)}{2}+ \frac{(2n+1-2n')(2n-2n')}{2}}.
\end{equation}
so 
\[
-f_1(2n+1,k)+f_2(n',k)-\frac{(k-2n')(k-2n'-1)}{2}+ \frac{(2n+1-2n')(2n-2n')}{2}= \frac{n'(n'+1)}{2}-n'(2n+1).
\]

Alternatively we can we can start from Identity 2 and set $2n = n' +j$ and see what formula results that corresponds to Identity 1: 
\begin{equation} \label{trial q identity 1}
\sum_{j=0}^{k-n'}(-1)^j q^{-j n'} { 2n'+j \choose 2n'}_q {k \choose n' + j}_q
\end{equation}
The above formula is equal to $\displaystyle (-1)^{n'-k} {n' \choose k-n'} $ at $q=1$, but for other $q$ it in general does not factor and is not equal to $ (-1)^{n'-k} {n' \choose k-n'}_q $ times some power of $q$. However, when $k=n'+1$ we do get 
\[
-q^{-n'} {n' \choose 1} 
\]
which actually follows from  \eqref{q identity 1 again}. That is, what \eqref{trial q identity 1} is missing to make it coincide with \eqref{q identity 1 again} is a factor of $q^{\frac{j(j-1)}{2}}$. Therefore perhaps \eqref{trial q identity 1} can be written as a sum of $q$-binomials, for example, to give another $q$-analogue of Identity 1. 

If we start from Identity 1 again and set $f_2(n,k)$ to be all 0, we get the sum for Identity 2 to be 
\[
\sum_{n'=0}^{2n} (-1)^{n'} q^{\frac{(2n-2n')(2n-2n'-1)}{2}}a_{n'}(q) \prod_{j=1}^{n'}(1+q^{j})^2{ 2n+n' \choose 2n'}_q 
\]
which does not completely factor either. 
 
If we try to bypass Identity 1 and compare the coefficients of $u^k$ directly, we get the identity: 
for each $k\geq 0$ 
\[
\sum_{n=0}^{\lfloor \frac{k}{2 } \rfloor} a_n {k \choose 2n} = \sum_{n=\lceil \frac{k}{2 } \rceil}^k (-1)^{n-k}a_n2^{2n} {n \choose k-n}.
\] 
We attempt a $q$-analogue of the above equation for $k=3$ with 
\[
1+q^a a_1(q) {3 \choose 2}_q = -q^b a_2(q)(1+q)(1+q^2 ){2 \choose 1}_q+ q^c a_3(q) (1+q)(1+q^2)(1+q^3)  
\]
where we have let $2^{2n}$ become $\displaystyle \prod_{j=1}^{n}(1+q^{j})^2$. It can be shown that this equation as a function of $q$ is not true for any real values of $a,b,$ and $c$. The same holds if we try to let $2^{2n}$ become $\displaystyle (1+q)^{2n}$ or just $2^{2n}$. 

\section{$q$-analogues and the Complete Elliptic Integral of the First Kind} \label{q first elliptic}
 Recall 
 \begin{align*}
 F(x) =& \frac{2}{\pi}\int_0^1\frac{dt}{\sqrt{1-t^2}\sqrt{1-xt^2}}\\ 
 =&\sum_{m=0}^\infty (\prod_{j=1}^m \frac{2j-1}{2j})^2 x^m.
 \end{align*}
 Therefore 
 \begin{align}
 \frac{1}{n!}\frac{d^n}{dx^n} F(x) |_{x=\frac{1}{2}} =& {-\frac{1}{2}\choose n} \frac{2}{\pi}\int_0^1\frac{1}{\sqrt{1-t^2}\sqrt{1-\frac{t^2}{2}}}\left( \frac{t^2}{1-\frac{t^2}{2}}\right)^n \, dt \label{integral} \\ 
 =&\sum_{m=0}^\infty (\prod_{j=1}^m \frac{2j-1}{2j})^2 {m \choose n}(\frac{1}{2})^{m-n} \label{sum}
 \end{align}

We present two $q$-analogues of the above formulas. In Section \ref{q sum} titled ``$q$-analogue of the Sum", we give a $q$-analogue of \eqref{sum}, which is actually phrased as a $q$ analogue 
\[
 \frac{1}{n!2^n}\frac{d^n}{dx^n} F(x) |_{x=\frac{1}{2}}=\sum_{m=0}^\infty (\prod_{j=1}^m \frac{2j-1}{2j})^2 {m \choose n}(\frac{1}{2})^{m}.
\]
In Section \ref{q integral}, titled ``$q$-analogue of the Integral", we give a $q$-analogue of \eqref{integral}. Despite the title of Section \ref{q integral}, we are actually giving a $q$-analogue of another sum that is obtained from that integral.   

\subsection{$q$-analogue of the Sum} \label{q sum}
We define a $q$-analogue of the function $1+\sin(\pi s)$ which we will use in Theorem \ref{q sum thm}.

\begin{definition}
\[
(1+\mathrm{SinPi})(s,q^2) = q^{s^2-s}\prod_{n=0}^\infty \frac{(1-q^{4n+3-2s})^2(1-q^{4n+1+2s})^2}{(1-q^{4n+3})^2(1-q^{4n+1})^2}.
\]
\end{definition}

\begin{theorem} 
The function $(1+\mathrm{SinPi})(s,q^2)$ is 2-periodic in $s$ and 
\[
\lim_{q \rightarrow 1^-} (1+\mathrm{SinPi})(s,q^2) = 1+\sin(\pi s).
\]
\end{theorem} 

\begin{proof}
The $2$-periodicity follows from
\[
(1+\mathrm{SinPi})(s+1,q^2)= q^{4s+2}\frac{(1-q^{-1-2s})^2}{(1-q^{1+2s})^2}  (1+\mathrm{SinPi})(s,q^2)= (1+\mathrm{SinPi})(s+1,q^2).
\]

Now we prove the limits $q \rightarrow 1^-$. We express 
\begin{align*}
(1+\mathrm{SinPi})(s,q^2)  =& q^{s^2-s}\prod_{n=0}^\infty \frac{(1-q^{4n+3-2s})^2(1-q^{4n+1+2s})^2}{(1-q^{4n+3})^2(1-q^{4n+1})^2} \\ 
=& \frac{1}{(\frac{-1-2s}{4})!_{q^4}^2 (\frac{-3+2s}{4})!_{q^4}^2}  C_1(q)
\end{align*}
where 
\[
C_1(q)=  \frac{(1-q^4)^2}{(1-q)^2}\prod_{n=1}^\infty \frac{(1-q^{4n})^2}{(1-q^{4n-1})^2 (1-q^{4n+1})^2}.
\]

As $q\rightarrow 1^-$, 
\[
\lim_{q\rightarrow 1^-} \frac{1}{(\frac{-1-2s}{4})!_{q^4}^2 (\frac{-3+2s}{4})!_{q^4}^2}  = \frac{1}{(\frac{-1-2s}{4})!^2 (\frac{-3+2s}{4})!^2} 
\]
which is equal to 
\[
\frac{\sin( \frac{\pi(2s+1)}{4})^2}{\pi^2}
\]
where we have used 
\[
(-x)! x! = \frac{\pi x}{\sin(\pi x)}.
\]
And we have
\begin{equation}
\lim_{q\rightarrow 1^-} C_1(q) = 2 \pi^2
\end{equation}
by Lemma \ref{q Wallis 4}.

Finally we have 
\[
2\sin( \frac{\pi(2s+1)}{4})^2= 1+\sin(\pi s)
\]
from standard trigonometric identities.
\end{proof}

Next we prove the limit in the previous lemma. It is a $q$-analogue of a product similar to the Wallis product for $\pi$.
\begin{lemma} \label{q Wallis 4}
\begin{equation} 
\lim_{q\rightarrow 1^-}C_1(q) =\lim_{q\rightarrow 1^-} \frac{(1-q^4)^2}{(1-q)^2}\prod_{n=1}^\infty \frac{(1-q^{4n})^4}{(1-q^{4n-1})^2 (1-q^{4n+1})^2} = 2 \pi^2.
\end{equation}
\end{lemma} 
\begin{proof}
We have  
\[
\frac{\sin(\pi x)}{\pi x} = \prod_{n=1}^\infty (1-\frac{x^2}{n^2})
\]
for all $x$. 
Setting $x = \frac{1}{4}$ and then taking the reciprocal gives 
\begin{equation} 
\prod_{n=1}^\infty \frac{(4n)^2}{(4n-1)(4n+1)} = \frac{\pi}{\sqrt{8}}.
\end{equation}
Let $L$ denote the limit 
\[
L = \lim_{q\rightarrow 1^-} \prod_{n=1}^\infty \frac{(1-q^{4n})^2}{(1-q^{4n-1}) (1-q^{4n+1})}. 
\]
We claim that $L= \frac{\pi}{\sqrt{8}}$. We claim that for each integer $n \geq 1$
\[
f(q,n)=\frac{(1-q^{n})^2}{(1-q^{n-1}) (1-q^{n+1})} 
\]
is an increasing function of $q$ for $q \in (0,1)$. That $\frac{\partial}{\partial q} f(q,n)\geq 0$ for $q \in (0,1)$ is equivalent to 
\[
(n-1)\frac{(1-q^{n+1})}{1-q} - q(n+1)\frac{(1-q^{n-1})}{1-q} \geq 0 
\] 
for $q \in (0,1)$. The above expression is equal to 
\[
(n-1)(1 + q^{n}) -2 \sum_{j=1}^{n-1}q^j = \sum_{j=1}^{n-1} (1-q^j)(1-q^{n-j}) \geq 0
\]
for $q \in (0,1)$, where we have used 
\[
\frac{(1-q^{N})}{1-q} = \sum_{j=0}^{N-1} q^j.
\]
Since 
\[
f(0,4n) =1 \,\,\,\,\,\mathrm{ and } \,\,\,\,\,f(1,n) =  \frac{(4n)^2}{(4n-1)(4n+1)},
\]
we can bound the limit $L$ between 
\begin{equation} 
\prod_{n=1}^N  \frac{(4n)^2}{(4n-1)(4n+1)} \leq L \leq  \prod_{n=1}^\infty  \frac{(4n)^2}{(4n-1)(4n+1)} = \frac{\pi}{\sqrt{8}} 
\end{equation}
for any $N$. 
This completes the proof.
\end{proof} 

Now we can prove the $q$-analogue of \eqref{sum}.

\begin{theorem} \label{q sum thm}
For any $s \in \mathbb{C}$ and $q \in (0,1]$,  

 \begin{align*}
 &\sum_{n=0}^\infty q^{(s-n)(s-n-1)}{n \choose s}_{q^2}(\prod_{j=1}^n \frac{1-q^{2j-1}}{1-q^{2j}})^2 \frac{1}{\prod_{j=1}^n(1+q^{2j})} \\ 
 =&  \, \, q^{s^2-s}\prod_{n=0}^\infty \frac{(1-q^{4n+3-2s})^2}{(1-q^{4n+2})(1-q^{4n+4})} \frac{(1-q^{2s+2n+2})}{(1-q^{2n+2})}\\ 
 =& \,\,C_2(q)(1+\mathrm{SinPi})(s,q^2) {s-\frac{1}{2} \choose s}_{q^2} \frac{\Gamma(\frac{2s+1}{4})_{q^4}}{\Gamma(\frac{2s+3}{4})_{q^4}}
\end{align*}
where 
\[
C_2(q) = (1-q^4)^{-\frac{1}{2}} \prod_{n=0}^\infty \frac{1-q^{2n+1}}{1-q^{2n+2}}
\]
and $C_2(1) = \frac{1}{\sqrt{2\pi}}$. 
\end{theorem}

\begin{proof}
We first prove the theorem for $q \in (0,1)$. From the definition of the $q$-binomial coefficient for non-integer $s$, we use 
\[
{n \choose s}_{q^2} = \prod_{j=1}^n(1-q^{2j})\prod_{j=n+1}^\infty (1-q^{2j-2s})\prod_{j=1}^\infty \frac{(1-q^{2j+2s})}{(1-q^{2j})^2}
\] 
to see that the theorem is equivalent to 
\[
\sum_{n=0}^\infty q^{n(n+1)-2sn}  \prod_{j=1}^n \frac{(1-q^{2j-1})^2}{(1-q^{4j})} \prod_{j=n+1}^\infty (1-q^{2j-2s})= \prod_{n=0}^\infty (1-q^{4n+3-2s})^2.
\]
To this equation we multiply both sides by $\displaystyle \prod_{j=1}^\infty (1-q^{4j})$ and set $x = q^{-2s}$ to get 
\begin{equation}\label{multiply by q^4 prod}
\sum_{n=0}^\infty q^{n(n+1)}x^n  \prod_{j=1}^n (1-q^{2j-1})^2 \prod_{j=n+1}^\infty (1-q^{2j}x)(1-q^{4j})= \prod_{n=0}^\infty (1-q^{4n+3}x)^2(1-q^{4n+4}).
\end{equation}
Let
\[
\tilde{f}(x,q) = \sum_{n=0}^\infty q^{n(n+1)}x^n(\prod_{j=1}^n (1-q^{2j-1})^2 ) (\prod_{j=n+1}^\infty (1-q^{2j}x)(1-q^{4j}) ).
\]
Then 
\[
\tilde{f}(x,q) = \prod_{n=0}^\infty  (1-q^{4n+4})(1-q^{4n+3}x)^2.
\]
We prove that 
\[
\tilde{f}(x,q) = (1-q^3x)^2 \tilde{f}(q^4x,q).
\]
Let 
\[
\tilde{f}_{n}(x,q) = q^{n(n+1)}x^n(\prod_{j=1}^n (1-q^{2j-1})^2 ) (\prod_{j=n+1}^\infty (1-q^{2j}x)(1-q^{4j}) ) 
\]
so 
\[
\tilde{f}(x,q) = \sum_{n=0}^\infty \tilde{f}_n(x,q).
\]
We calculate that 
\begin{align*}
\tilde{f}_{n}(x,q) -(1-q^3 x)^2 \tilde{f}_n(q^4 x,q) = q^{n(n+1)}x^n(1-q^{4n} - x(q^{2n+2}+ q^{2n+4}-2q^{4n+3}) )\\ 
\times  \prod_{j=1}^n (1-q^{2j-1})^2 \prod_{j=n+3}^\infty (1-q^{2j}x) \prod_{j=n+1}^\infty (1-q^{4j} ). 
\end{align*}
We claim that 
\[
\sum_{n=0}^N \tilde{f}_{n}(x,q) - \tilde{f}_n(q^4 x,q)(1-q^3 x)^2 = -q^{(N+1)(N+2)}x^{N+1} \prod_{j=1}^{N+1} (1-q^{2j-1})^2\prod_{j=N+3}^\infty (1-q^{2j}x) \prod_{j=N+1}^\infty (1-q^{4j}). 
\]
We prove this claim by induction on $N$. It is true for $N=0$. Assume it is true for some $N\geq 0$.
Then the induction step is implied by the identity
\[
-(1-q^{2N+6})(1-q^{4N+4})+1-q^{4N+4}-x(q^{2N+4} +q^{2N+6}-2q^{4N+7}) = -xq^{2N+4}(1-q^{2N+3})^2.
\]
Therefore 
\[
\lim_{N \rightarrow \infty} \sum_{n=0}^N \tilde{f}_{n}(x,q) -(1-q^3 x)^2 \tilde{f}_n(q^4 x,q) =0
\]
and 
\[
\tilde{f}(x,q) =  \tilde{f}(q^4x,q)(1-q^3x)^2.
\]
Iterating gives 
\[
\tilde{f}(x,q) = \tilde{f}(0,q)\prod_{n=0}^\infty (1-q^{4n+3}x)^2
\]
and 
\[
\tilde{f}(0,q) = \prod_{n=0}^\infty (1-q^{4n+4}).
\]
This proves the theorem for $q \in (0,1)$.

To prove it for $q=1$, note that if $s$ is a non-negative integer, all sums and products become finite, so we may take the limit $q\rightarrow 1^{-}$ and we are done. If $s$ is a negative integer, then each term in the sum is 0 and the right hand side is also 0. 

If $s$ is not an integer, we follow the same procedure for $q<1$, but, instead of multiplying by $\displaystyle \prod_{j=1}^\infty (1-q^{4j})$ at \eqref{multiply by q^4 prod}, we divide by $\displaystyle \prod_{j=1}^\infty (1-q^{2j}x)$. With $x = q^{-2s}$, we let $q=1$ and set 
\begin{equation*}
f_n(s) =\prod_{j=1}^n \frac{(2j-1)^2}{(2j-2s)(4j)} \,\,\,\, \mathrm{ and } \, \, \, \, f(s) = \sum_{n=0}^\infty f_n(s).  
\end{equation*}
Lemma \ref{sum f(s) convergent} proves the convergence of the sum $f(s)$. 

We prove that 
\[
f(s) = \frac{(3-2s)^2}{(2-2s)(4-2s)}f(s-2).
\]
By the same reasoning for $q<1$, we have 
\[
\sum_{n=0}^{N} (f_n(s) - \frac{(3-2s)^2}{(2-2s)(4-2s)}f_n(s-2)) = -\frac{(2N+1)^2}{(2N+2-2s)(2N+4-2s)}\prod_{j=1}^N \frac{(2j-1)^2}{(2j-2s)(4j)}.
\]
The right side of the above equation goes to 0 as $N\rightarrow \infty$ by the same reasoning we give for the bounds of $f_n(s)$ in Lemma \ref{sum f(s) convergent}.
Therefore 
\[
\lim_{N\rightarrow \infty } \sum_{n=0}^{N}( f_n(s) - \frac{(3-2s)^2}{(2-2s)(4-2s)}f_n(s-2))=0
\]
proving 
\[
f(s) = \frac{(3-2s)^2}{(2-2s)(4-2s)}f(s-2).
\]
Iterating we have 
\[
f(s) = (\lim_{N\rightarrow \infty} f(s-N))\prod_{n=0}^\infty \frac{(4n+3-2s)^2}{(4n+2-2s)(4n+4-2s)}
\]
and 
\[
 \lim_{N\rightarrow \infty} f(s-N)=1
\]
because for $N >\mathrm{Re}(s)$, we have 
\[
|f(s-N) -1|\leq 2^{-1}\frac{K}{\sqrt{(1+|\mathrm{Re}(s-N)|)^2+ \mathrm{Im}(s)^2}}
\]
from the proof of Lemma \ref{sum f(s) convergent}. This proves the theorem for $q=1$. 

Now 
\[
{s-\frac{1}{2} \choose s}_{q^2} =\prod_{n=0}^\infty \frac{(1-q^{2s+2n+2})(1-q^{2n+1})}{(1-q^{2s+1+2n})(1-q^{2s+1+4n})}
\]
and 
\[
\frac{\Gamma(\frac{2s+1}{4})_{q^4}}{\Gamma(\frac{2s+3}{4})_{q^4}} = (1-q^4)^{\frac{1}{2}}\prod_{n=0}^\infty \frac{1-q^{2s+3+4n}}{1-q^{2s+1+4n}}.
\]
Therefore we can express the right side as 
\[
(1+\mathrm{SinPi}) (s,q^2) {s-\frac{1}{2} \choose s}_{q^2} C(q)\frac{\Gamma(\frac{2s+1}{4})_{q^4}}{\Gamma(\frac{2s+3}{4})_{q^4}} 
\]
where 
\[
C_2(q) = (1-q^4)^{-\frac{1}{2}} \prod_{n=0}^\infty \frac{1-q^{2n+1}}{1-q^{2n+2}}.
\]
Now 
\[
\lim_{q\rightarrow 1^-} C_2(q) = \frac{1}{\sqrt{pi}} 
\]
because $C_2(q)^2$ is a $q$-analogue for the Wallis product of $\pi$; the limit follows from similar reasoning in Lemma \ref{q Wallis 4} by taking $x=\frac{1}{2}$ in the product for $\sin(\pi x)$. 
\end{proof}

\begin{lemma} \label{sum f(s) convergent}
For $s$ not a positive integer, the sum $f(s)$ is convergent, where
 \begin{equation*}
f_n(s) =\prod_{j=1}^n \frac{(2j-1)^2}{(2j-2s)(4j)} \,\,\,\, \mathrm{ and } \, \, \, \, f(s) = \sum_{n=0}^\infty f_n(s).  
\end{equation*}
\end{lemma}
\begin{proof}
The sum on the right is convergent because if $\mathrm{Re}(s)>0$ then we may bound $|f_n(s)|$ by 
\[
|f_n(s)|\leq 2^{-n}p_s(n)
\] 
where $p_s(n)$ is a polynomial in $n$ whose coefficients and degree depend on $s$. To see this, we have for $n> \mathrm{Re}(s)$ 
\begin{align}
|f_n(s)|=&|2^{-n}\prod_{j=2}^n \frac{(2j-1)^2}{(2j-2)(2j)}| \label{convergent product}\\ 
&\times |\frac{1}{4n}\prod_{j=1}^{\lceil \mathrm{Re}(s) \rceil -1}\frac{j+n-\lceil \mathrm{Re}(s) \rceil+1}{j-s} | \label{poly depending on s}\\ 
&\times |\prod_{j=1}^{n-\lceil \mathrm{Re}(s) \rceil +1}\frac{j}{j+\lceil \mathrm{Re}(s) \rceil - \mathrm{Re}(s)+i\mathrm{Im}(s) }| \label{bounded by 1}. 
\end{align}
Now the product at \eqref{convergent product} is convergent as $n \rightarrow \infty$; the product at \eqref{poly depending on s} is bounded by a polynomial in $n$ depending on $s$, and \eqref{bounded by 1} is

\[
 |\prod_{j=1}^{n-\lceil \mathrm{Re}(s) \rceil +1}\frac{j}{j+\lceil \mathrm{Re}(s) \rceil - \mathrm{Re}(s)+i\mathrm{Im}(s) }| \leq  \prod_{j=1}^{n-\lceil \mathrm{Re}(s) \rceil +1}\frac{1}{\sqrt{(1+\lceil \mathrm{Re}(s) \rceil - \mathrm{Re}(s))^2 + \mathrm{Im}(s)^2 }}
\]
which is bounded by 1. 

 If $\mathrm{Re}(s)<0$, then by the above reasoning we may bound $f_n(s)$ for $n\geq 1$ by 
\[
|f_n(s)|\leq 2^{-n}\frac{K}{\sqrt{(1+|\mathrm{Re}(s)|)^2+ \mathrm{Im}(s)^2}}.
\]
where $K$ is a constant independent of $s$.
\end{proof}

We include this lemma which be useful elsewhere.
\begin{lemma}
Let $t,s,q \in \mathbb{R}$ such that $t\geq s$. Let
\[
f(s,t,q) = \frac{(1-q^t)}{(1-q^{t-s})}.
\]
Then for fixed $t$ and $s$, $f(t,s,q)$ is an increasing function of $q$ on $(0,1)$. 
\end{lemma}
\begin{proof}
Taking $\frac{\partial}{\partial q} f(s,t,q)$, we see that the lemma is equivalent to 
\[
-s+t +sq^t-tq^s \geq 0.
\]
for $q \in (0,1)$. This is equivalent to   
\[
g(t,q) = \frac{(1-q^t)}{t}
\]
being a decreasing function of $t$ for $t\in \mathbb{R}$ and for fixed $q$. To prove that $g(t,q)$ is a decreasing function, choose $\alpha >0$ and write by the binomial expansion 
\begin{align*}
q^t = (1-(1-q^{\alpha}))^{\frac{t}{\alpha}} = &\sum_{n=0}^\infty (-1)^n \frac{(1-q^\alpha)^n}{n!\alpha^n} \prod_{j=0}^{n-1} (t-\alpha j)\\
& =1-t\sum_{n=1}^\infty  \frac{(1-q^\alpha)^n}{n!\alpha^n} \prod_{j=1}^{n-1} (\alpha j-t).
\end{align*}  
Therefore 
\[
g(t,q) = \sum_{n=1}^\infty  \frac{(1-q^\alpha)^n}{n!\alpha^n} \prod_{j=1}^{n-1} (\alpha j-t).
\]
For $t \in (-\infty, \alpha)$, each term in the above sum is positive and, for $n\geq 2$, decreases in magnitude as $t$ increases to $\alpha$, while the $n=1$ term remains constant. Therefore $g(t,q)$ is decreasing on $(-\infty,\alpha)$ for any $\alpha >0$.  
\end{proof}

\subsection{$q$-analogue of the Integral} \label{q integral}
Now we prove a $q$-analogue of \eqref{integral}. We first show how to obtain a sum from the integral.
\begin{lemma}
\[
\int_0^1 \frac{1}{\sqrt{1-t^2}\sqrt{2-t^2}}(\frac{t^2}{2-t^2})^s \, dt  =
\frac{(-\frac{1}{2})! (s-\frac{1}{2})!}{2(s!)} \sum_{n=0}^\infty (-1)^n \prod_{j=1}^n \frac{(2j-1)(2s+2j-1)}{(2j)(2s+2j)}
 \] 
\end{lemma} 
\begin{proof}
We make the change of variable $t \mapsto \sqrt{t}$ and express $\displaystyle \frac{1}{\sqrt{2-t}}(\frac{t}{2-t})^s$ as a binomial series in $1-t$ to obtain 
\[
\int_0^1 \frac{1}{\sqrt{1-t^2}\sqrt{2-t^2}}(\frac{t^2}{2-t^2})^s \, dt  = \sum_{n=0}^\infty \int_0^1t^{s-\frac{1}{2}}(1-t)^{n-\frac{1}{2}} { -s-\frac{1}{2} \choose n}\, \frac{dt}{2}.
\]
To this we apply 
\[
\int_0^1 at^{a-1}(1-t)^b \, dt = \frac{a!b!}{(a+b)!}
\]
and 
\[
(x!)(-x)! = \frac{\pi x}{\sin(\pi x)}
\]
to obtain 
\[
\frac{1}{2} \sum_{n=0}^\infty  (-1)^n \frac{(n-\frac{1}{2})!}{n!} \frac{(s+n-\frac{1}{2})!}{(s+n)!} 
\]
which is equal to 
\[
\frac{(-\frac{1}{2})! (s-\frac{1}{2})!}{2(s!)} \sum_{n=0}^\infty (-1)^n \prod_{j=1}^n \frac{(2j-1)(2s+2j-1)}{(2j)(2s+2j)}.
\]
\end{proof}

We present a $q$-analogue of the above sum and its evaluation as a product in the following theorem: 

\begin{theorem}
Let 
\[
f(x,q) = \sum_{n=0}^\infty (-1)^n q^n \prod_{j=1}^n \frac{(1-q^{2j-1}x)(1-q^{+2j-1})}{(1-q^{2j})(1-q^{2j}x)}.
\]
Then for $s \in \mathbb{C}$ not a negative integer and $q \in (0,1]$, 
\begin{align*}
f(q^{2s},q) = &f(0,q)\prod_{n=0}^\infty \frac{(1- q^{2s+4n+3})^2}{(1-q^{2s+4n+2})(1-q^{2s+4n+4})}\\ 
               =& f(0,q)\frac{C_3(q)}{(1+q^2)^s}\frac{(s!)_{q^2}}{(\frac{2s-1}{4})!_{q^4}^2} 
\end{align*}
where 
\[
C_3(q) =\left(\frac{(1-q^4)}{(1-q^2)} \prod_{n=1}^\infty \frac{(1-q^{4n})^2}{(1-q^{4n-2})(1-q^{4n+2})}\right)^{\frac{1}{2}}
\]
and $C_3(1) = \sqrt{\pi}$ and $f(0,1)= \frac{1}{\sqrt{2}}$. 
That is, 
\begin{align*}
&\sum_{n=0}^\infty (-1)^n q^n \prod_{j=1}^n \frac{(1-q^{2j-1})(1-q^{2s+2j-1})}{(1-q^{2j})(1-q^{2s+2j})} \\ 
=& \left(\sum_{n=0}^\infty (-1)^n q^n \prod_{j=1}^n \frac{(1-q^{2j-1})}{(1-q^{2j})}\right)  
\prod_{n=0}^\infty \frac{(1- q^{2s+4n+3})^2}{(1-q^{2s+4n+2})(1-q^{2s+4n+4})}. 
\end{align*}

\end{theorem}
\begin{proof}
We first prove theorem for $q \in (0,1)$. We multiply the left side of the theorem by 
\begin{equation} \label{divide}
\prod_{j=1}^\infty  (1-q^{2j})(1-q^{2s+2j})
\end{equation}
and set $x=q^{2s}$ to obtain 
\[
\tilde{f}(x,q)=\sum_{n=0}^\infty \tilde{f}_n(x,q). 
\]
where
\[
\tilde{f}_n(x,q) = (-1)^n q^n(\prod_{j=1}^n(1-q^{2j-1})(1-q^{2j-1} x))(\prod_{j=n+1}^\infty(1-q^{2j})(1-q^{2j} x)) 
\]

We prove
\[
\tilde{f}(x,q)=(1-q^3 x)^2 \tilde{f}(q^4x, q). 
\]
We claim
\begin{multline}\label{induction}
\sum_{n=0}^N (\tilde{f}_n(x,q) -(1-q^3 x)^2 \tilde{f}_n(q^4x, q)) =q^2(1-q) xR_N(x,q)(\prod_{j=N+1}^\infty (1-q^{2j}))(\prod_{j=N+3}^\infty (1-q^{2j}x))
\end{multline} 
where 
\[
R_N(x,q)= (-1)^N q^{N} (\prod_{m=0}^N( 1-q^{2m+1}))(\prod_{m=1}^N( 1-q^{2m+1}x)). 
\]
We prove \eqref{induction} by induction on $N$. It is true for $N=0$ and $1$. Assume it is true for $N\geq 1$.
Then  
\begin{align*}
\sum_{n=0}^{N+1} (\tilde{f}_n(x,q) -(1-q^3 x)^2 \tilde{f}_n(q^4x, q)) &= q^2(1-q) xR_N(x,q)  (\prod_{m=N+1}^\infty (1-q^{2m}))(\prod_{m=N+3}^\infty (1-q^{2m}x))\\ 
&+ \tilde{f}_{N+1}(x,q) -(1-q^3 x)^2 \tilde{f}_{N+1}(q^4x, q).
\end{align*}
Now
\begin{align*}
\tilde{f}_{N+1}(x,q) -&(1-q^3 x)^2 \tilde{f}_{N+1}(q^4x, q) = -qR_N(x,q)\\
&\times((1-q^3 x)(1-q^{2N+3}x)(1-q^{2N+5}x) -(1-q x)(1-q^{2N+4}x)(1-q^{2N+6}x) )\\
&\times(\prod_{m=N+2}^{\infty}(1-q^{2m}))(\prod_{m=N+4}^{\infty}(1-q^{2m} x)).
\end{align*}
To this we apply the identity
\begin{align*}
& q^2(1-q)x(1-q^{2N+2})(1-q^{2N+6}x) \\ 
&- q ((1-q^3 x)(1-q^{2N+3}x)(1-q^{2N+5}x) -(1-q x)(1-q^{2N+4}x)(1-q^{2N+6}x) ) \\ 
&=  -(1-q)q^3 x (1-q^{2N+3})(1-q^{2N+3}x).
\end{align*}
Therefore 
\begin{align*}
\sum_{n=0}^{N+1} (\tilde{f}_n(x,q) -(1-q^3 x)^2 \tilde{f}_n(q^4x, q)) &= -q^3(1-q) x  R_N(x,q) (1-q^{2N+3})(1-q^{2N+3}x)  \\ 
&\times (\prod_{m=N+2}^\infty (1-q^{2m}))(\prod_{m=N+4}^\infty (1-q^{2m}x))\\ 
&=  q^2(1-q) x R_{N+1}(x,q) (\prod_{m=N+2}^\infty (1-q^{2m}))(\prod_{m=N+4}^\infty (1-q^{2m}x)).
\end{align*}
This completes the induction step. 

Because of the $q^N$ in $R_N(x,q)$, we have
 \[
 \lim_{N \rightarrow \infty} R_N(x,q)=0, 
\]
so
\[
\sum_{n=0}^\infty (\tilde{f}_n(x,q) - (1-q^3x)^2 \tilde{f}_n(q^4x,q))=0
\]
and thus 
\[
\tilde{f}(x,q) = (1-q^3x)^2 \tilde{f}(q^4x,q).
\]
Iterating we obtain
\[
\tilde{f}(x,q)=\tilde{f}(0,q)\prod_{m=0}^\infty (1- q^{4m+3}x)^2.
\]
Now we divide both sides by 
\[
\prod_{j=1}^\infty (1-q^{2j})(1-q^{2j}x)
\] 
which completes the proof for $q\in (0,1)$. 

For $q=1$, we follow the same procedure for $q<1$, but do not divide by \newline $\displaystyle \prod_{j=1}^\infty  (1-q^{2j})(1-q^{2s+2j})$ at \eqref{divide}. We let 
\[
f_n(s) = (-1)^n \prod_{j=1}^{n}  \frac{(2j-1)(2s+2j-1)}{(2j)(2s+2j)}) \, \, \,\, \mathrm{ and } \,\,\,\,\, f(s) = \sum_{n=0}^\infty f_n(s).
\]
In Lemma \ref{second q=1 convergence} we prove that the sum on the right is convergent. We now prove 
\begin{equation} \label{second f transform}
f(s) = \frac{(3+2s)^2}{(2+2s)(4+2s)}f(s+2).
\end{equation}
From the above reasoning for $q<1$ we have 
\[
\sum_{n=0}^N  f_n(s) - \frac{(3+2s)^2}{(2+2s)(4+2s)}f_n(s+2) = (-1)^N \frac{2N+2}{(2N+2+2s)(2N+4+2s)}\prod_{j=1}^N \frac{(2j-1)(2j+1+2s)}{(2j)(2j+2s)}.
\]
The limit of the above sum as $N \rightarrow \infty$ is 0 because the product on the right converges as $N \rightarrow \infty$ for any $s$ not a negative integer. This proves \eqref{second f transform}.
Iterating we have 
\[
f(s) = (\lim_{N \rightarrow \infty} f(s+N) ) \prod_{n=0}^\infty  \frac{(4n+3+2s)^2}{(4n+2+2s)(4n+4+2s)}
\]
where $\displaystyle \lim_{N \rightarrow \infty} f(s+N) $ is determined in Lemma \ref{second q=1 convergence}. 

The expression of the product using $q$-factorials follows from their definition. $C_3(1)=\sqrt{\pi}$ follows from the same reasoning used for the limit of $C_2(q)$. And we can determine 
\[
f(0,1) = \frac{1}{\sqrt{2}}
\]
by comparison with the evaluation found in Theorem \ref{q sum}.
\end{proof}

\begin{lemma} \label{second q=1 convergence}
The sum 
\[
f(s)=\sum_{n=0}^\infty (-1)^n \prod_{j=1}^n  \frac{(2j-1)(2s+2j-1)}{(2j)(2s+2j)}
\]
is convergent 
and 
\[
\lim_{N \rightarrow \infty} f(s+N) = \sum_{n=0}^\infty (-1)^n \prod_{j=1}^n  \frac{(2j-1)}{(2j)} = \frac{1}{\sqrt{2}}.
\]

\end{lemma}
\begin{proof} 
First we prove that 
\[
\sum_{n=0}^\infty (-1)^n \prod_{j=1}^n  \frac{(2j-1)(2s+2j-1)}{(2j)(2s+2j)}
\]
is convergent for $s$ not a negative integer. We group the $2n$-th and $(2n+1)$-th terms together to express the sum as 
\begin{equation} \label{q elliptic 2 q=1}
\sum_{n=0}^\infty (\prod_{j=1}^{2n}  \frac{(2j-1)(2s+2j-1)}{(2j)(2s+2j)})(\frac{(3+8n+2s)}{(4n+2)(4n+2+2s)}).
\end{equation}
We compare \eqref{q elliptic 2 q=1} to the sum when $s=0$:
\begin{equation} \label{q elliptic 2 q=1 s=0}
\sum_{n=0}^\infty (\prod_{j=1}^{2n}  \frac{(2j-1)}{(2j)})^2(\frac{(3+8n)}{(4n+2)(4n+2)}) = \sum_{n=0}^\infty \frac{{4n \choose 2n}^2}{2^{8n}}\frac{(3+8n)}{(4n+2)(4n+2)}
\end{equation}
Using Stirling's approximation 
\[
n! \sim \sqrt{2 \pi n} (\frac{n}{e})^n 
\]
we have 
\[
 \frac{{4n \choose 2n}^2}{2^{8n}}\frac{(3+8n)}{(4n+2)(4n+2)} \sim \frac{1}{2\pi n^2}.
\]
Therefore \eqref{q elliptic 2 q=1 s=0} is convergent. If $s<0$ and $s \in \mathbb{Z}+\frac{1}{2}$, then the sum \eqref{q elliptic 2 q=1} is finite. For other $s$, we apply the limit comparison test to to sums \eqref{q elliptic 2 q=1} and \eqref{q elliptic 2 q=1 s=0} to get 
\[
\lim_{n \rightarrow \infty } | \frac{(4n+2)(2s+3+8n)}{ (3+8n)(2s+4n+2)} \prod_{j=1}^{2n}  \frac{(2j)(2s+2j-1) }{(2j-1)(2s+2j)}|.
\]
This infinite product is convergent to a non-zero number because the sum
\[
\sum_{j=1}^\infty (1- \frac{(2j)(2s+2j-1)}{ (2j-1)(2s+2j)}) = \sum_{j=1}^\infty \frac{s}{(2j-1)(s+j)}
\]
is convergent. Therefore \eqref{q elliptic 2 q=1} is convergent for any $s \in \mathbb{C}$ not a negative integer. 

We claim 
\begin{equation} \label{N lim second q elliptic}
\lim_{N\rightarrow \infty} \sum_{n=0}^\infty (\prod_{j=1}^{2n}  \frac{(2j-1)(2s+2N+2j-1)}{(2j)(2s+2N+2j)})(\frac{(3+8n+2s+2N)}{(4n+2)(4n+2+2s+2N)})  = \sum_{n=0}^\infty (\prod_{j=1}^{2n}  \frac{(2j-1)}{(2j)})\frac{1}{(4n+2)} 
\end{equation}
First we have that if $a,b \in \mathbb{R}, j \in \mathbb{Z}$ with $a$ and $j>0$, then 
\[
\left |\frac{a+bi +2j-1}{a+bi+2j}\right | = \sqrt{1+\frac{-2a-4j+1}{(a+2j)^2+b^2}}\leq 1. 
\]
Therefore in \eqref{N lim second q elliptic}, using $a+bi = 2s+2N$, we assume that $N$ is so large that $\mathrm{Re}(N+s)>0$.
Next, the sum on the right of \eqref{N lim second q elliptic} is convergent using Stirling's approximation again, so for any $\epsilon>0$ we can choose $n_1$ such that
\[
\sum_{n=m}^\infty (\prod_{j=1}^{2n}  \frac{(2j-1)}{(2j)})\frac{1}{(4n+2)} <\epsilon.
\]
for all $m > n_1$ and also such that 
\[
\left | \frac{(3+8n+2s+2N)}{(4n+2+2s+2N)} \right |< 3
\]
for all $n > n_1$ with $\mathrm{Re}(N+s)>0$. Thus we have 
\begin{align*}
\lim_{N \rightarrow \infty}|\sum_{n=0}^{\infty } &(\prod_{j=1}^{2n}  \frac{(2j-1)(2s+2N+2j-1)}{(2j)(2s+2N+2j)})(\frac{(3+8n+2s+2N)}{(4n+2)(4n+2+2s+2N)}) - (\prod_{j=1}^{2n}  \frac{(2j-1)}{(2j)})\frac{1}{(4n+2)} | \\ 
 &\leq   4\epsilon + \lim_{N \rightarrow \infty} |\sum_{n=0}^{n_1 } (\prod_{j=1}^{2n}  \frac{(2j-1)(2s+2N+2j-1)}{(2j)(2s+2N+2j)})(\frac{(3+8n+2s+2N)}{(4n+2)(4n+2+2s+2N)}) \\ 
&- (\prod_{j=1}^{2n}  \frac{(2j-1)}{(2j)})\frac{1}{(4n+2)} | \\ 
&=4\epsilon.
\end{align*}
 This proves the claim \eqref{N lim second q elliptic}. As mentioned in Theorem \ref{q integral} the sum $f(0) = \frac{1}{\sqrt{2}}$ by comparison with Theorem \ref{q sum}. This completes the proof. 
\end{proof}

\section{Further Work} \label{further work}

\begin{itemize}

\item See if there are $q$-analogues of other proofs of the arithmetic-geometric mean functional equation. 

\item See if $q$-analogues can be found for the arithmetic-geometric mean applied to complex numbers.

\item Find $q$-analogues for generalizations of the geometric-mean such as the cubic counterpart in \cite{Borwein}. 

\item Try to reconcile Identities 1 and 2 to construct a $q$-analogue of the functional equation itself, possibly using more than one function. 
 
\item Use $q$-analogues of $F(x)$ to determine $q$-analogues of $k(x)^2$ and thus $\theta_3(e^{\pi x})^2$. 
 
 For this point, $k(x)^2$ is the function
 \[
 k(x)^2= 1 -\frac{ \theta_3(e^{\pi x})^4}{ \theta_4(e^{\pi x})^4}.
 \]
 Now $k(x)^2$ is also determined by the properties 
 \begin{equation}\label{k 1/x}
 k(x)^2 + k(\frac{1}{x})^2=1
 \end{equation}
 and 
 \begin{equation} \label{k x}
 x F(k(x)^2) = F(1-k(x)^2).
 \end{equation}
 That is, those two properties imply
 \[
 \frac{\theta_4^4(e^{-\pi x})}{\theta_3^4(e^{-\pi x})} -\frac{1}{2}= 
 \sum_{n=0}^\infty (-1)^n  \frac{\overline{\kappa}(n)}{(2n+1)!}  (\frac{(\frac{\Gamma(\frac{1}{4})}{\Gamma(\frac{3}{4})})^4}{4})^{2n+1} (\frac{x-1}{x+1})^{2n+1}
 \]
 where $\overline{\kappa}(n)$ is the sequence $\{ 1,6,104, 3024, 130176, 7831296, ...\}$.
 Therefore a $q$-analogue of $F(x)$ can by used to define a $q$-analogue of $k(x)^2$ via \eqref{k 1/x} and \eqref{k x}. Then a $q$-analogue of $k(x)^2$ can be used to define a $q$-analogue of  $\theta_3(e^{\pi x})^2$ by 
 \begin{equation} \label{theta3 and k}
 \theta_3(e^{\pi x})^2 = F(1-k(x)^2).
 \end{equation}
 We note that $\theta_3(q)$ can itself be viewed as arising from a $q$-analogue of $\sin(\pi x)$, so above we are talking about a $q$-analogue of a function that is a specialization (at $q= e^{-\pi}$) of a $q$-analogue of another function ($\sin(x)$). 

We also note that the Mellin transform of $ \theta_3(e^{\pi x})^2-1$ is a $\Gamma$ function factor times 
\begin{equation}\label{zeta L-4}
\zeta(s) L_{-4}(s)
\end{equation}
where 
\[
L_{-4}(s) = \sum_{n=0}^\infty \frac{(-1)^{n}}{(2n+1)^s}.
\]
Therefore considering $\theta_3(e^{\pi x})^2$ directly may be easier than considering $\theta_3(e^{\pi x})$ and would contain information about $\zeta(s)$ and its zeros. Studying the coefficients of $k(x)^2$ or its $q$-analogues could yield information of the generalized Tur\'an inequalities for \eqref{zeta L-4} or an expression of the coefficients as elementary-symmetric polynomials. 

\item The Mellin transform \eqref{zeta L-4} follows from a Lambert series for  $ \theta_3(e^{\pi x})^2$. Find a combinatorial proof of this identity. 

\item Equation \eqref{theta3 and k} is actually combinatorial identity. Find an explicit combinatorial proof of this identity and see if it has a $q$-analogue.

\item See if $q$-analogues and infinite product evaluations exist for elliptic integrals of the second kind. 

\item The coefficients $a_n$  are 
\[
a_n= (\prod_{j=1}^n \frac{2j-1}{2j})^2 = \frac{{2n\choose n}^2}{2^{4n}}
 \]
 where we may interpret $\displaystyle {2n\choose n}^2$ as the number of lattice paths on a square grid that start at one corner and go to the opposite corner and then return. Find out how Identity 2 translates into operations on these lattice paths.

\end{itemize}

\end{document}